\documentclass[11pt]{article}

\usepackage{amssymb}
\usepackage{mathrsfs}
\usepackage{latexsym,amsmath,amssymb,amsfonts,epsfig,graphicx,cite,psfrag}
\usepackage{eepic,color,colordvi,amscd}
\usepackage{ebezier}
\usepackage{verbatim}
\usepackage{subfigure}
\usepackage{accents}
\usepackage{amsthm}
\usepackage{comment}
\usepackage{tikz}
\usepackage[colorlinks=true,
linkcolor=blue,citecolor=blue,
urlcolor=blue]{hyperref}

\theoremstyle{plain}
\newtheorem{theo}{Theorem}[section]

\newtheorem{lem}[theo]{Lemma}

\newtheorem{conj}[theo]{Conjecture}

\theoremstyle{definition}

\theoremstyle{remark}

\title{Number of Hamiltonian cycles in planar triangulations}
\author{Xiaonan Liu and Xingxing Yu\footnote{Partially supported by NSF Grant DMS 1954134}\medskip
\\ School of Mathematics\\Georgia Institute of Technology\\Atlanta, GA 30332}
\date{}

\begin{document}
\maketitle
\begin{abstract}

Whitney proved in 1931 that 4-connected planar triangulations are Hamiltonian.  Hakimi, Schmeichel, and Thomassen conjectured in 1979 that if $G$ is a 4-connected planar triangulation with $n$ vertices then $G$ contains at least $2(n-2)(n-4)$ Hamiltonian cycles, with equality if and only if $G$ is a double wheel. 
On the other hand, 
a recent result of  Alahmadi, Aldred, and Thomassen states that there are exponentially many Hamiltonian cycles in 
5-connected planar triangulations. In this paper, we consider 4-connected planar $n$-vertex triangulations $G$ that do not have too many separating 4-cycles or have minimum degree 5. We show that if $G$ has  $O(n/{\log}_2 n)$ separating 4-cycles then $G$ has $\Omega(n^2)$ Hamiltonian cycles, and  if $\delta(G)\ge 5$  then $G$ has $2^{\Omega(n^{1/4})}$ Hamiltonian cycles. Both results improve previous work. Moreover, the proofs involve a
``double wheel'' structure, providing further evidence to the above conjecture. 
\end{abstract}

\bigskip


\bigskip 


\newpage

\section{Introduction}
A cycle $C$ in a connected graph $G$ is said to be {\it separating} 
if the graph obtained from $G$ by deleting $C$  is not connected. For any positive integer $k$, a {\it $k$-cycle} is a cycle of length $k$. 
A separating 3-cycle is also known 
as a {\it separating triangle}.  A {\it planar triangulation} is a plane graph in which every face is bounded by a triangle (i.e., 3-cycle).

In 1931, Whitney \cite{Wh31}  showed that every planar triangulation without separating triangles is Hamiltonian. In 1956, Tutte \cite{Tu56} proved that every 4-connected planar graph is Hamiltonian.  
Thomassen \cite{Th83} showed in 1983 that every 4-connected planar graph is {\it Hamiltonian connected}, i.e.,  for any distinct vertices $x$ and $y$ there exists a Hamiltonian path between $x$ and $y$. Thus, every 4-connected planar graph has more than one Hamiltonian cycle. 

Hakimi, Schmeichel, and Thomassen \cite{HST79} proved in 1979 that every 4-connected planar triangulation has at least $n/\log_2 n$ Hamiltonian cycles. Recently, Brinkmann, Souffriau and Van Cleemput \cite{BSV18} improved the lower bound to $\frac{12}{5}(n-2)$. 
Consider a {\it double wheel}, a planar triangulation obtained from a cycle by adding two vertices and all edges from these two vertices to the cycle. 
Observe that a double wheel with $n$ vertices has  precisely $2(n-2)(n-4)$ Hamiltonian cycles. Hakimi, Schmeichel, and Thomassen \cite{HST79} conjectured that, among 4-connected planar triangulations, the  double wheels have the smallest number of Hamiltonian cycles. 

\begin{conj}[Hakimi, Schmeichel, and Thomassen, 1979]\label{conjecture}
If $G$ is a 4-connected planar triangulation with $n$ vertices, then $G$ has at least $2(n-2)(n-4)$ Hamiltonian cycles, with equality if and only if $G$ is a double wheel.
\end{conj}

This conjecture remains open and appears to be difficult. There are natural related questions one can ask: What can we say about the number of Hamiltonian cycles in  5-connected planar triangulations? What about 4-connected planar triangulations without many separating 4-cycles  or with minimum degree at least 5?

Recently, Lo \cite{Lo20} showed that every 4-connected $n$-vertex planar triangulation with $O(\log n)$ separating 4-cycles has $\Omega((n/\log n)^2)$ Hamiltonian cycles. In this paper,  we improve Lo's result by weakening its hypothesis and strengthening its conclusion to a quadratic bound.

\begin{theo}\label{main1}
Let $G$ be a 4-connected planar triangulation with $n$ vertices and $O( n/\log_2 n)$ separating 4-cycles.  Then $G$ has $\Omega (n^2)$ Hamiltonian cycles. 
\end{theo}


Alahmadi, Aldred, and Thomassen \cite{AAT20}  proved that every 5-connected $n$-vertex planar triangulation has $2^{\Omega (n)}$ Hamiltonian cycles, improving the earlier bound $2^{\Omega(n^{1/4})}$ of B\"{o}hme, Harant, and Tk{\' a}{\v c} \cite{BHT99}. We prove the following result which improves the result of B\"{o}hme, Harant, and Tk{\' a}{\v c} by replacing the 5-connectedness condition with``minimum degree 5".

\begin{theo}\label{main2}
Let $G$ be a 4-connected planar triangulation with $n$ vertices and  minimum degree $5$. Then $G$ has $2^{\Omega( n^{1/4})}$ Hamiltonian cycles.
\end{theo}

In Section 2, we discuss a key idea in \cite{AAT20} used to show the existence of exponentially many Hamiltonian cycles in a 5-connected planar triangulation. We also  collect several known results on ``Tutte paths'' in planar graphs, and use them to see when a certain planar graph has at least two Hamiltonian paths between two given vertices.   

In Section 3, we prove Theorem~\ref{main1}. Basically, we show that if a  4-connected planar triangulation $G$ does not have too many separating 4-cycles, then either $G$ has a large independent set with nice properties, or $G$ has two vertices with many common neighbors (i.e., $G$ has a large structure which resembles a double wheel). In either case, we can find the desired number of Hamiltonian cycles in $G$.

In Section 4, we prove Theorem~\ref{main2}. We will see that if a 4-connected planar triangulation $G$ has minimum degree  $5$, then either $G$ has a large independent set with nice properties, or $G$ has two vertices with a lot of common neighbors, or $G$ has many separating 4-cycles.  For the first two possibilities, we use similar arguments as in the proof of Theorem~\ref{main1}. For the third possibility, we show that there are many separating 4-cycles in $G$ which either have pairwise disjoint interiors or are all pairwise  ``nested". In both cases, we can find many Hamiltonian cycles in $G$.  

\medskip 

We conclude this section with some terminology and notation. Let $G$ and $H$ be graphs. We use $G\cup H$ and $G\cap H$ to denote the union and intersection of $G$ and $H$, respectively. For any $S\subseteq V(G)$, we use $G[S]$ to denote the subgraph of $G$ induced by $S$, and let $G-S$ denote the graph obtained from $G$ by deleting $S$ and all edges of $G$ incident with $S$. A set $S\subseteq V(G)$ is a {\it cut} in $G$ if $G-S$ has more components than $G$, and if $|S|=k$ then $S$ is a cut of {\it size $k$} or {\it $k$-cut} for short.  For a subgraph $T$ of $G$, we often write $G-T$ for $G-V(T)$ and write $G[T]$ for $G[V(T)]$. 
A path (respectively, cycle) is often represented as a sequence (respectively, cyclic sequence) of vertices with consecutive vertices 
being adjacent.  Given a path $P$ and distinct vertices $x,y\in V(P)$, we use $xPy$ to denote the subpath of $P$ between $x$ and $y$.

Let $G$ be a graph. For each $v\in V(G)$, we use $N_G(v)$ to denote the neighborhood of $v$ in $G$, and if there is no confusion we omit the reference to $G$. 
If $H$ is a subgraph of $G$, we write  $H \subseteq G$. For any set $R$ consisting of 1-element or 2-element subsets of $V(G)$, we use $H+R$ to denote the graph with vertex set $V(H)\cup (R\cap V(G))$ and edge set $E(H)\cup (R\setminus V(G))$. If $R=\{\{x,y\}\}$ (respectively, $R=\{v\}$), we write $H+xy$ (respectively, $H+v$)  instead of $H+R$.

Let $G$ be a plane graph. The {\it outer walk} of $G$ consists of vertices and edges of $G$ incident with the infinite face of $G$. If the outer walk is a cycle in $G$, we call it  {\it outer cycle} instead.  If all vertices of $G$ are incident with its infinite face, then we say that $G$ is an {\it outer planar} graph.  For a cycle $C$ in  $G$, we use $\overline{C}$ to denote the subgraph of $G$ consisting of all vertices and edges of $G$ contained in the closed disc bounded by $C$. The {\it interior} of $C$ is then defined as the subgraph $\overline{C}-C$. For any distinct vertices $u,v\in V(C)$, we use $uCv$ to denote the subpath of $C$ from $u$ to $v$ in clockwise order.

\section{Preliminaries}

Alahmadi, Aldred, and Thomassen \cite{AAT20} recently proved that if $G$ is a 5-connected $n$-vertex planar triangulation then $G$ has  an independent set $S$ of $\Omega(n)$ vertices, such that $G-F$  is 4-connected for each set $F$ consisting of $|S|$ edges of $G$ that are incident with $S$. There are  $2^{\Omega (n)}$ choices of $F$. Hence, applying the above mentioned theorem of Tutte to each $G-F$, it follows from a simple calculation that  $G$ has $2^{\Omega (n)}$ Hamiltonian cycles. 

How could a cut of size at most 3 occur after removing from a 4-connected planar triangulation such a set $F$ of edges incident with an independent set $S$? Alahmadi {\it et al}  observed that this could happen if a vertex in $S$ is contained in a separating 4-cycle, or a vertex in $S$ is  adjacent to three vertices of a separating 4-cycle, or two vertices in $S$ are contained in a separating 5-cycle, or three vertices in $S$ occur in some  9-vertex graph called diamond-6-cycle. A {\it diamond-6-cycle} is a graph isomorphic to the graph shown on the left in Figure~\ref{diamond cycle}, in which the vertices of degree 3 are called {\it crucial} vertices.  We also define {\it diamond-4-cycle} here for later use; it is a graph isomorphic to the graph shown on the right in Figure~\ref{diamond cycle}, where the two degree 3 vertices not adjacent to the degree 2 vertex are its {\it crucial} vertices. 
\begin{figure}[htbp] 
\centering
\includegraphics[width=10cm,height=4.5cm]{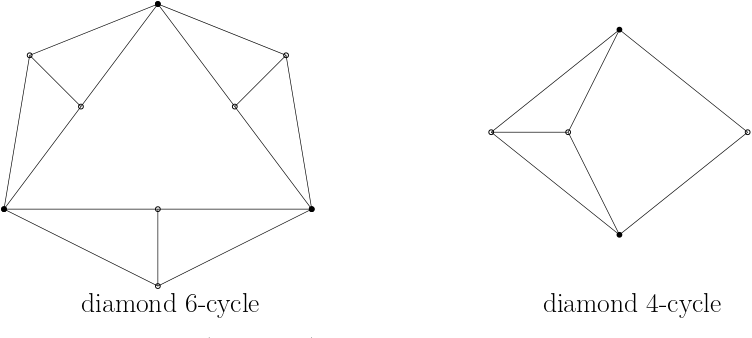}
  \caption{diamond-6-cycle and diamond-4-cycle} 
\label{diamond cycle}
\end{figure}

Formally,  let $S$ be an independent set in a graph $G$. We say that $S$ {\it saturates} a 4-cycle or 5-cycle $C$ in $G$ if $|S\cap V(C)|=2$, and  $S$ {\it saturates} a diamond-6-cycle $D$ in $G$ if $S$ contains three crucial vertices of $D$.

The following result for 5-connected planar triangulations was proved by Alahmadi {\it et al} \cite{AAT20}. Lo \cite{Lo20} observed that it is essentially true for certain  4-connected planar triangulations. We further observe that a slight variation holds for 4-connected planar triangulations with minimum degree 5. We provide a proof here as it is short and gives the key idea used in \cite{AAT20} for proving the existence of exponentially many Hamiltonian cycles in 5-connected planar triangulations. 

\begin{lem}\label{edgesetF}
Let $S$ be an independent set in a 4-connected planar triangulation $G$ with $|V(G)|\ge 6$, and assume that the following conditions hold:
\begin{itemize}
\item[\textup{(i)}] all vertices in $S$ have degree at most 6,
\item[\textup{(ii)}] $S$ saturates no 4-cycle, or 5-cycle, or diamond-6-cycle,
\item[\textup{(iii)}] $\delta(G)\ge 5$, or no vertex in $S$ is contained in a separating 4-cycle, and
\item[\textup{(iv)}] no vertex in $S$ is adjacent to 3 vertices of any separating 4-cycle.
\end{itemize}
Let $F$ be any subset of $E(G)$ with $|F|=|S|$ edges such that for each vertex $v\in S$, there is exactly one edge in $F$ incident with $v$. Then $G-F$ is 4-connected and has a
Hamiltonian cycle. Moreover, if ${\cal C}$ is a collection of Hamiltonian cycles in $G$ formed by taking precisely one Hamiltonian cycle from $G-F$ for each choice of $F$, then 
$|{\cal C}|\ge (3/2)^{|S|}$.
\end{lem}
\begin{proof}

First, suppose  $G-F$ is 4-connected for all possible choices of $F$. Then $S$ contains no vertex of degree 4 in $G$; so by (i), all vertices in $S$ have degree 5 or 6 in $G$. Let $a_1$ and $a_2$ denote the number of vertices in $S$ of degree 5 and 6, respectively. There are $5^{a_1}6^{a_2}$ choices of the edge set $F\subseteq E(G)$, with precisely one edge from each vertex in $S$. For each choice of $F$, $G-F$ has a Hamiltonian cycle by a result of Tutte. Let ${\cal C}$ be a collection of Hamiltonian cycles in $G$ obtained by taking precisely one Hamiltonian cycle from each $G-F$. Then each Hamiltonian cycle of $G$ in ${\cal C}$ is chosen at most $(5-2)^{a_1}(6-2)^{a_2}=3^{a_1}4^{a_2}$ times. Thus $|{\cal C}|\ge (5/3)^{a_1}(6/4)^{a_2}\geq (3/2)^{a_1+a_2}=(3/2)^{|S|}$.

\medskip

Now suppose there exists an $F$ such that $G-F$ is not 4-connected. Let $K$ be a minimal cut of $G-F$; so $|K|\leq 3$. Let $G_1,G_2$ be subgraphs of $G-F$ 
such that $G-F=G_1\cup G_2$, $V(G_1\cap G_2)=K$, $E(G_1\cap G_2)=\emptyset$, and $V(G_i)\ne K$ for $i=1,2$. 
Let $F'$ be the set of the edges between $G_1-K$ and $G_2-K$ in $G$. Then $F'\subseteq F$.  Since $G$ is 4-connected,  $G-K$ is connected; so $F'\ne \emptyset$. 

Since $G$ is a 4-connected planar triangulation, for each $e\in F'$, the two vertices incident with $e$ have exactly two common neighbors, which must be contained in $K$. Hence, $|K|\geq 2$. 

Also observe that, for any two edges $e_1,e_2\in F'$,  there do not exist distinct vertices $u,v\in K$ such that all vertices incident with $e_1$ or $e_2$ are contained in $N_G(u)$ and $N_G(v)$. For, otherwise, the vertices $u$ and $v$ form a 4-cycle with  the two vertices in $S$  that are incident with $e_1$ or $e_2$,  contradicting (ii). 

By the above observations, $|F'|\le  \binom{|K|}{2}$. Moreover, $|K|=3$ as otherwise $|K|=2$ and $|F'|\leq \binom{2}{2}=1$, contradicting the assumption that $G$ is 4-connected. Hence, $1\le |F'|\le 3$. 

Suppose $|F'|=1$ and let $uv\in F'$ with $u\in S$. Then $G[K\cup \{v\}]$ or $G[K\cup \{u\}]$ is a separating 4-cycle in $G$ (as $|V(G)|\geq 6$). Now,  $G[K\cup \{v\}]$ is not a separating 4-cycle in $G$; otherwise, $G[K\cup \{v\}]$ has three neighbors of $u$, contradicting (iv). Then $G[K\cup \{u\}]$ is a separating 4-cycle in $G$. Thus, $d_G(v)=4$ and the vertex $u\in S$ is contained in a separating 4-cycle in $G$, which contradicts (iii).

If $|F'|=2$, then let $u,v\in S$ be incident with the edges in $F'$. We see that $G[K\cup \{u,v\}]$ contains a 5-cycle, contradicting (ii). 
So $|F'|=3$, and let $u,v,w\in S$ be incident with the edges in $F'$. Since $S$ saturates no 4-cycle by (ii), $F'$ is a matching in $G$. But then we see that $G[K\cup \{u,v,w\}]$ contains a diamond-6-cycle in which $u, v, w$ are crucial vertices, contradicting (ii). 
\end{proof}

We need the following result from Lo \cite{Lo20}.

\begin{lem}[Lo, 2020] \label{5-cycle}
Let $G$ be a 4-connected planar triangulation and let $S$ be an independent set of vertices of degree at most 6 in $G$,  such that $S$ saturates no 4-cycle in $G$. Then there exists a subset $S'\subseteq S$ of size at least $|S|/541$ such that $S'$ saturates no 5-cycle in $G$.
\end{lem}

Lo \cite{Lo20}  also observed that  the following lemma stated for 5-connected planar triangulations in \cite{AAT20} actually holds for 4-connected planar triangulations. 

\begin{lem}[Alahmadi, Aldred, and Thomassen, 2020; Lo, 2020] \label{diamond}
Let G be a 4-connected planar triangulation and let $S$ be an independent set of vertices of degree at most 6 in $G$,  such that $S$ saturates no 4-cycle in $G$. Then there exists a subset $S'\subseteq S$ of size at least $|S|/301$ such that $S'$ saturates no diamond-6-cycle in $G$.
\end{lem}

Lo \cite{Lo20} proved a lemma implying that any 4-connected planar triangulation has a large independent set or contains two vertices with a lot of common neighbors.

\begin{lem}[Lo, 2020] \label {ISET}
Let $G$ be a 4-connected planar triangulation. Let $S$ be an independent set of vertices of degree at most 6, and $S'$ be a maximal subset of $S$ such that $S'$ saturates no 4-cycle in $G$. Then there exist two distinct vertices $v,x\in V(G)$ such that $|N(v)\cap N(x)| \geq |S|/(9|S'|)$.
\end{lem}

 We use Lemma~\ref{ISET} to derive the following result, which will be applied by setting $t=\lfloor c\log_2 n \rfloor $ or $t=\lfloor c n^{1/4} \rfloor$ for some constant $c>0$.

\begin{lem}\label{4-cycle}
Let $G$ be a 4-connected planar triangulation with $n$ vertices. For any positive integer $t$, one of  the following holds: 
\begin{itemize}
    \item[\textup{(i)}] There exist two distinct vertices 
    $v,x\in V(G)$ such that $|N(v)\cap N(x)| >t$.
\item[\textup{(ii)}] There is an independent set $S$ of vertices of degree at most 6 in $G$, such that  $S$ saturates no 4-cycle in $G$ and $|S|\geq
n/(108t)$. 
\end{itemize}
\end{lem}
\begin{proof}
Since each vertex of $G$ has degree at least 4 and $|E(G)|=3n-6$ by Euler's formula, there exist at least $n/3$ vertices of degree at most 6 in $G$. Therefore, by the Four Color Theorem, $G$ has an independent set $I$ of vertices of degree at most 6, such that $|I|\geq n/12$.

 Let $S$ be a maximal subset of $I$ such that $S$ saturates no 4-cycle in $G$. If $|S|\ge n/(108t)$ then (ii) holds. So assume $|S|<n/(108t)$.  
By Lemma~\ref{ISET}, there exist $v\neq x\in V(G)$ such that $|N(v)\cap N(x)|\geq |I| /(9|S|)\geq (n/12) /(9n/108t)>t$.

\end{proof}
Note that, when Lemma~\ref{4-cycle} is applied later, we always have $t\geq 2$ and $v, x$ are non-adjacent as $G$ is 4-connected. 

\medskip

We now show that if $G$ does not have too many separating 4-cycles then the independent set in Lemma~\ref{4-cycle} may be required to satisfy additional properties.

\begin{lem}\label{special-set}
Let $G$ be a 4-connected planar triangulation with $n$ vertices and at most $c_1 n/\log_2 n$ separating 4-cycles, where $c_1=(108\times 16\times 541\times 301\times 2)^{-1}$. Then one of the following holds:
\begin{itemize}
\item[\textup{(i)}] There exist non-adjacent  vertices $v,x\in V(G)$ such that $|N(v)\cap N(x)|> \lfloor 16\log_2 n \rfloor$. 
\item [\textup{(ii)}]$G$ has an independent set $S$ with $|S|\ge c_1 n/\log_2 n$ such that 
\subitem{\textup{(a)}}  all vertices in $S$ have degree at most 6,
\subitem{\textup{(b)}} $S$ saturates no 4-cycle, or 5-cycle, or diamond-6-cycle,
\subitem{\textup{(c)}} no vertex in $S$ is contained in a separating 4-cycle, and
\subitem{\textup{(d)}} no vertex in $S$ is adjacent to 3 vertices of a separating 4-cycle.
\end{itemize}
\end{lem}
\begin{proof}
Suppose (i) does not hold. Then by Lemma~\ref{4-cycle},  $G$ has an independent set $S_1$ of vertices of degree at most 6, such that 
$S_1$ saturates no 4-cycle in $G$ and $|S_1|\geq n/(108\times 16 \log_2 n )$.
By Lemma~\ref{5-cycle}, there exists $S_2\subseteq S_1$ such that  $|S_2|\geq |S_1|/541$ and $S_2$ saturates no 4-cycle or 5-cycle in $G$. By Lemma~\ref{diamond}, there exists $S_3\subseteq S_2$ such that 
$$|S_3|\geq |S_2|/301\geq |S_1|/(541\times 301)\geq 2c_1 n/\log_2 n$$ 
and $S_3$ saturates no 4-cycle, or 5-cycle, or diamond-6-cycle in $G$. Thus $S_3$ satisfies (a) and (b).

To obtain $S\subseteq S_3$ such that $S$ satisfies (c) and (d), we show that, for any separating 4-cycle $C$ in $G$, $|V(C)\cap S_3|+|T(C)|\le 1$, where $T(C):=\{v\in S_3\setminus V(C):  |N(v)\cap V(C)|\ge 3\}$. Let $C$ be an arbitrary separating 4-cycle in $G$. Note that  $|V(C)\cap S_3|\le 1$, since $S_3$ is independent  and satisfies (b). Also note that  $|T(C)|\le 1$; for,  any two vertices in $T(C)$ are contained in a 4-cycle, a contradiction as $S_3$ satisfies  (b).  Moreover, if $|V(C)\cap S_3|=1$ then $|T(C)|=0$; for, any $v\in T(C)$ and the vertex in $V(C)\cap S_3$ are contained in a 4-cycle, a contradiction as $S_3$ satisfies (b). Hence, $|V(C)\cap S_3|+|T(C)|\le 1$. 

Let $S=S_3\setminus \bigcup_{C} ((V(C)\cap S_3)\cup T(C))$. Then $S$ satisfies (c) and (d), in addition to (a) and (b). Since $G$ has at most $c_1 n/\log_2 n$ separating 4-cycles, $|\bigcup_{C}((V(C)\cap S_3)\cup T(C))|\le c_1 n/\log_2 n$.  Hence, $|S|\geq |S_3|-c_1 n/\log_2n \geq c_1n/\log_2 n$. So  (ii) holds.
\end{proof}

From time to time, we need to find at least two Hamiltonian paths between two given vertices in a  subgraph of a planar triangulation. For this, we need several results on ``Tutte paths'' in planar graphs which 
are defined using the notion of ``bridge''. 
 Let $G$ be a graph and $H\subseteq G$. An {\it $H$-bridge} of $G$ is a subgraph of $G$ induced by either an edge in $E(G)\setminus E(H)$ with both incident vertices in $V(H)$, or all edges in $G-H$ with at least one incident vertex in a single component of $G-H$.
For an  $H$-bridge $B$ of $G$, the vertices in $V(B\cap H)$ are the {\it attachments} of $B$ on $H$.

A path $P$ in a graph $G$ is called a {\it Tutte path} if every $P$-bridge of $G$ has at most three attachments on $P$. If in addition, every $P$-bridge of $G$ containing an edge of some subgraph $C$ of $G$ has at most two attachments on $P$, then $P$ is called a {\it $C$-Tutte path} in $G$.
When proving that 4-connected planar graphs are Hamiltonian connected, Thomassen \cite{Th83} proved a stronger result on Tutte paths in 2-connected planar graphs. 

\begin{lem}[Thomassen, 1983] \label{tuttepath}
Let $G$ be a 2-connected plane graph and $C$ be its outer cycle, and let $x\in V(C)$,  $y\in V(G)\setminus\{x\}$, and $e\in E(C)$. Then $G$ has a $C$-Tutte path $P$ between $x$ and $y$ such that $e\in E(P)$. 
\end{lem}

Note that if the graph $G$ in Lemma~\ref{tuttepath} has no 2-cut  contained in $V(C)$ and no 3-cut separating $C$ from some vertex in $V(G)\setminus V(C)$ and $e\neq xy$, then the path $P$ is in fact a Hamiltonian path between $x$ and $y$ in $G$. Later when we say that ``by Lemma~\ref{tuttepath}, we find a Hamiltonian path $P$'' we are actually using this observation. 

A {\it near triangulation} is a plane graph in which all faces except possibly its infinite face are bounded by triangles.  We now derive a simple result on the number of Hamiltonian paths between two given vertices in near triangulations.

\begin{lem}\label{uw-path}
Let $G$ be a near triangulation with outer cycle $C:=uvwxu$ and assume that $G\ne C+vx$ and  $G$ has no separating triangles. Then one of the following holds:
\begin{itemize}
\item[\textup{(i)}] $G-\{v, x\}$ has at least two Hamiltonian paths between $u$ and $w$.
\item[\textup{(ii)}] $G-\{v,x\}$ is a path between  $u$ and $w$ and, hence, outer planar.
\end{itemize}
\end{lem}
\begin{proof}
If $vx\in E(G)$, then $G=C+vx$ or  $G$ has a separating triangle, contradicting our assumption.  So $vx\notin E(G)$.  Then $G-\{v,x\}$ has a path from 
$u$ to $w$, say $Q$.  Since $G$ has no separating triangles, each block of $G-\{v,x\}$ contains an edge of $Q$. Hence, the blocks of $G-\{v,x\}$ can be labeled as $B_1, \ldots, B_t$ and the cut vertices of $G-\{v,x\}$ can be labeled as $b_1,\ldots, b_{t-1}$ such that $V(B_i\cap B_{i+1})=\{b_i\}$ for $i=1, \ldots, t-1$, and $V(B_i\cap B_j)=\emptyset$ when $|i-j|\ge 2$.  Let $b_0=u$ and $b_t=w$. Moreover, let $C_i$ denote the outer walk of $B_i$ for $1\le i\le t$. See Figure~\ref{blocks}.

\begin{figure}[htbp] 
\centering
\includegraphics[width=11.5cm,height=5.5cm]{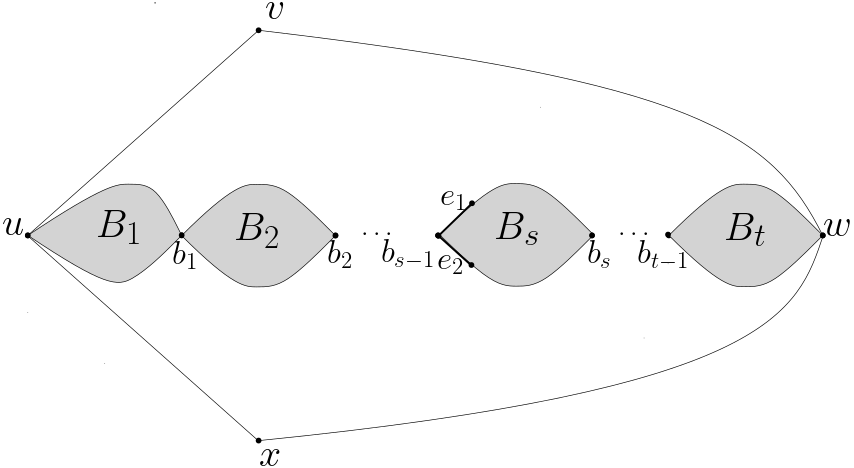}
  \caption{The blocks $B_1, \ldots, B_t$.} 
\label{blocks}
\end{figure}
If $|V(B_i)|=2$ for  $1\leq i\leq t$, then (ii) holds. Hence, we may assume that $|V(B_s)|\geq 3$ for some $s$, where $1\leq s\leq t$. Then $b_{s-1}b_s\notin E(B_s)$, as otherwise, $vb_{s-1} b_sv$ or $xb_{s-1}b_sx$ would be a separating triangle in $G$. 
Let $e_1, e_2$ be the edges of $C_s$ incident with $b_{s-1}$. By Lemma~\ref{tuttepath}, $B_s$ has a $C_s$-Tutte path $P_s^j$ between  $b_{s-1}$ and  $b_s$ such that  $e_j\in E(P_s^j)$, for $j=1,2$. Since $G$ has no separating triangles, $P_s^1$ and $P_s^2$ are  Hamiltonian paths in $B_s$.

For each $1\leq i\leq t$ with $i\ne s$, if $|V(B_i)|\geq 3$, we apply Lemma~\ref{tuttepath} to $B_i$ and find a Hamiltonian path $P_i$ between  $b_{i-1}$ and $b_i$ in $B_i$; if $|V(B_i)|=2$, let $P_i=b_{i-1}b_i$. Then $(\bigcup_{i\ne s}P_i)\cup P_s^1$ and $(\bigcup_{i\ne s}P_i)\cup P_s^2$ are  distinct Hamiltonian paths in $G-\{v,x\}$ between $u$ and $w$. So (i) holds.  
\end{proof}

\medskip

We also  need the following result of Thomas and Yu \cite{TY94}, which was used to extend Tutte's theorem on Hamiltonian cycles in planar graphs to projective planar graphs.

\begin{lem}[Thomas and Yu, 1994] \label{tuttepath-1}
Let $G$ be a 2-connected plane graph with outer cycle $C$, and let $u,v\in V(C)$ and $e,f\in E(C)$ such that $u, e, f, v$ occur on $C$ in clockwise order. Then $G$ has a $uCv$-Tutte path $P$ between $u$ and $v$ such that $e,f\in E(P)$.
\end{lem}

We now use Lemma~\ref{tuttepath-1} to prove a result similar to Lemma~\ref{uw-path}. 

\begin{lem} \label{uv-path}
Let $G$ be a near triangulation with outer cycle $C: =uvwxu$ and assume that $G$ has no separating triangles. Then one of the following holds:
\begin{itemize}
   \item [\textup{(i)}]  $G-\{w,x\}$ is an outer planar near triangulation.  
    \item [\textup{(ii)}]  $G-\{w,x\}$ has at least two Hamiltonian paths between $u$ and $v$.
\end{itemize}
\end{lem}
\begin{proof}
 We apply induction on $|V(G)|$. If $|V(G)|=4$ then we see that (i) holds trivially. So assume $|V(G)|\ge 5$. Then $uw,vx\notin E(G)$, as $G$ has no separating triangles.

We may assume that $u,v$ each have at least two neighbors in $V(G)\setminus V(C)$.  For, otherwise, by symmetry assume that $u$ has a unique neighbor in $V(G)\setminus V(C)$, say $u'$. Now $G':=G-u$ is a near triangulation with outer cycle $C':=u'vwxu'$ and $G'$ has no separating triangles. Hence, by induction, $G'-\{w,x\}$ is an outerplanar near triangulation, or $G'-\{w,x\}$ has at least two Hamiltonian paths between  $u'$ and $v$. In the former case, (i) holds; in the latter case, (ii) holds by extending the Hamiltonian paths in $G'$ from $u'$ to $u$ along the edge $u'u$.

Next, we claim that $(G-\{w,x\})-u$ or $(G-\{w,x\})-v$ is 2-connected.  For, suppose  $(G-\{w,x\})-u$ is not 2-connected. Then  $(G-\{w,x\})-u$  can be written as the union of two subgraphs $B_1$ and $B_2$ such that $|V(B_1\cap B_2)|\le 1$, $B_1-B_2\ne \emptyset$, and $B_2-B_1\ne \emptyset$. Without loss of generality, assume that $v\in V(B_2)$. (Indeed, $v\in V(B_2)\setminus V(B_1)$.)
  We further choose $B_1,B_2$ to minimize $B_1$. Then $B_1$ is connected and $B_1$ has no cut vertex.   
By planarity, there exists a unique vertex $y\in N_G(w)\cap N_G(x)$.  If $y\in V(B_2)$ then $V(B_1\cap B_2)\cup \{u,x\}$ is a 2-cut in $G$ or induces a separating triangle in $G$, a contradiction.  So $y\in V(B_1)\setminus V(B_2)$. Now $u$ has a neighbor in $V(B_1)\setminus V(B_2)$; as otherwise, $V(B_1\cap B_2)\cup \{w,x\}$  is a 2-cut in $G$ or induces a separating triangle in $G$, a contradiction. This implies that $G[B_1+u]$ is 2-connected.  Now, we repeat this argument for $(G-\{w,x\})-v$. Suppose $(G-\{w,x\})-v$ is not 2-connected. Then $(G-\{w,x\})-v$  can be written as the union of two subgraphs $B_1'$ and $B_2'$ such that $|V(B_1'\cap B_2')|\le 1$, $B_1'-B_2'\ne \emptyset$, $B_2'-B_1'\ne \emptyset$, and  $u\notin V(B_1')\setminus V(B_2')$. Then, since $G[B_1+u]$ is 2-connected and $y\in V(B_1)\setminus V(B_2)$, we have $y\in V(B_2')$. Now, $V(B_1'\cap B_2')\cup \{v,w\}$  is a 2-cut in $G$ or induces a separating triangle in $G$, a contradiction.

By symmetry, we may assume that $H:=(G-\{w,x\})-u$ is 2-connected.
Let $D$ denote the outer cycle of $H$ and let $u_1,u_2\in N_G(u)\cap V(D)$ such that $u_1\in N_G(x)$ and $u_2\in N_G(v)$. Since $u$ has at least two neighbors in $V(G)\setminus V(C)$, $u_1\ne u_2$. Let $y\in N_G(w)\cap N_G(x)$. Choose an edge $e\in E(D)$ incident with $y$, and an edge $f\in E(D)$ incident with $u_1$. By Lemma~\ref{tuttepath}, $H$ has a $D$-Tutte path $P$ between $u_1$ and $v$ such that $e\in E(P)$. By Lemma~\ref{tuttepath-1}, $H$ has a $vDu_2$-Tutte path $Q$ between $u_2$ and $v$ such that  $e,f\in E(Q)$. Since $G$ has no separating triangles, we see that both $P,Q$ are Hamiltonian paths in $H$. Now $P\cup u_1u $ and $Q\cup u_2u$ are distinct Hamiltonian 
paths in $G-\{w,x\}$ between $u$ and $v$, and (ii) holds.
\end{proof}

Later we will need the following result of Jackson and Yu \cite{JY02} on Hamiltonian cycles through more than two given edges in planar triangulations. This result  was used to show that planar triangulations with few separating triangles are Hamiltonian. 

\begin{lem}[Jackson and Yu, 2002] \label{4edges}
Let $G$ be a planar triangulation with no separating triangles. Let $T, T_1, T_2$ be distinct triangles in $G$. Let $V(T)=\{u,v,w\}$. Then there exists a Hamiltonian cycle $C$ in $G$ and edges $e_1\in E(T_1), e_2\in E(T_2)$ such that $uv,uw,e_1,e_2$ are dinstinct and contained in $E(C)$.
\end{lem}

\section{Planar triangulations with  few separating 4-cycles}

 In this section, we consider 4-connected planar triangulations without many separating 4-cycles, as a  natural relaxation of 5-connected planar triangulations. The main objective here is to show that the number of Hamiltonian cycles in such graphs is quadratic in the number of vertices. 
 First, we need the following result. 

\begin{lem}\label{2edge}
Let  $c_1=(108\times 16\times 541\times 301\times 2)^{-1}$, and let $G$ be a 4-connected planar triangulation with $n$ vertices and at most $c_1 n/\log_2 n$ separating 4-cycles. Then for any two edges $e, f$ in a triangle in $G$, 
there are at least $c_1^2 n$ Hamiltonian cycles in $G$ containing both $e$ and $f$.
\end{lem}


\begin{proof} We apply induction on the number of vertices in $G$. If $n\le 1/c_1^2$ then the result follows from Lemma~\ref{4edges} that every 4-connected planar triangulation has a Hamiltonian cycle containing two given edges in a triangle. We may thus assume that $n\ge 1/c_1^2$. 

 Let $e,f$ be two edges  of $G$ such that $e,f\in E(T)$ for some triangle $T$. Without loss of generality, we may assume that $T$ is the outer cycle of $G$. By Lemma~\ref{special-set},  there are non-adjacent vertices $v,x\in V(G)$ such that $|N(v)\cap N(x)|>
\lfloor16\log_2n \rfloor$,  or $G$ has an independent set $S$ of size at least $c_1 n/\log_2 n$ satisfying (a), (b), (c),  and (d) in Lemma~\ref{special-set}.

Suppose  $G$ contains an independent set $S$ of size at least $c_1 n/\log_2 n$ and satisfying (a), (b), (c),  and (d) in Lemma~\ref{special-set}. Then $S_1:=S\setminus V(T)$ is an independent set of size at least $c_1 n/\log_2 n-1$ as $S$ has at most 1 vertex in $V(T)$. 
Let $F\subseteq E(G)$ be obtained by choosing precisely one edge incident with each vertex in $S_1$; then $G-F$ is 4-connected by Lemma~\ref{edgesetF}. Let $V(T)=\{u,v,w\}$ such that $e=uv$ and $f=vw$. By Lemma~\ref{tuttepath}, $G-F$ has a $T$-Tutte path $P_F$ between $v$ and  $w$ and containing  the edge $e$. Since $G-F$ is 4-connected, $C_F:=P_F+f$ is a Hamiltonian cycle in $G-F$. Now form a collection ${\cal C}$ by,  for each choice of $F$, taking from  $G-F$ exactly one Hamiltonian cycle $C_F$ that contains $e$ and $f$. Then all cycles in $\cal{C}$ are Hamiltonian cycles in $G$ containing $e$ and $f$. By Lemma~\ref{edgesetF}, 
$$|{\cal C}|\ge   (3/2)^{|S_1|}\geq (3/2)^{(c_1n/\log_2 n)-1}\geq c_1^2 n,$$ 
since $n\ge 1/c_1^2$. 

Thus, we may assume that  there are non-adjacent vertices $v,x\in V(G)$ such that $|N(v)\cap N(x)|>
\lfloor 16\log_2n \rfloor$. Let $C=uvwxu$ such that  $N(v)\cap N(x)\subseteq V(\overline{C})$. Note that $\overline{C}$ is a near triangulation and $T\not\subseteq \overline{C}$ (as $T$ bounds the infinite face of $G$). Since $G$ is 4-connected, $\overline{C}- \{v,x\}$ has a path between $u$ and $w$, and every block of $\overline{C}-\{v,x\}$ contains an edge of that path. Hence, the  blocks of $\overline{C} -\{v,x\}$ can be labelled as $B_1, \ldots, B_t$ and the vertices in $N(v)\cap N(x)$ can be labelled as $u_0, u_1, \ldots, u_{t-1}, u_t$ such that $V(B_i\cap B_{i+1})=\{u_i\}$ for $1\le i\le t-1$, $B_i\cap B_j=\emptyset$ when $|i-j|\ge 2$, $u_0=u\in V(B_1-u_1)$,  and $u_{t}=w\in V(B_t-u_{t-1})$.

Observe that,  for  $1\leq j\leq t$, if $|V(B_j)|\ge 3$ then we use Lemma~\ref{tuttepath} to conclude that $B_j$ has at least two Hamiltonian paths between $u_{j-1}$ and $u_j$.
\medskip

{\it Case} 1. $|\{i: |V(B_i)|\ge 3\}|\ge 2\log_2 n$.

Let $G^*$ denote the graph obtained from $G$ by contracting $\overline{C}-C$ to a single vertex $v^*$.  Then $G^*$ is a 4-connected planar triangulation  (and $v^*$ has degree 4 in $G^*$).  Hence, by applying Lemma~\ref{tuttepath}, we see that 
 $G^*$ has a Hamiltonian cycle $C^*$ such that  $e, f\in E(C^*)$. 

If $uv^*, wv^*\in E(C^*)$, then the union of $C^*-v^*$ and a Hamiltonian path between $u$ and $w$ in $\bigcup_{i=1}^t B_i$ is a Hamiltonian cycle in $G$ containing both $e$ and $f$. By the above observation,  $\bigcup_{i=1}^tB_i$ has at least $2^{2\log_2n}=n^2$ Hamiltonian paths between $u$ and $w$. Hence, the number of Hamiltonian cycles in $G$ containing both $e$ and $f$ is at least $n^2\geq c_1^2 n$.

Now assume $uv^*, vv^*\in E(C^*)$. Then the union of $C^*-v^*$ and a Hamiltonian path between $u$ and $v$ in $\overline{C}-\{w,x\}$ is a Hamiltonian cycle in $G$. By the above observation, 
$\bigcup_{i=1}^{t-1}B_i$ has at least $2^{2\log_2 n-1}=n^2/2$ Hamiltonian paths between $u$ and $u_{t-1}$.  By Lemma~\ref{tuttepath}, $G[(B_t-w)+v]$ has a Hamiltonian path between $u_{t-1}$ and $v$ (and containing $N_G(w)\cap N_G(v)$). The union of any such two paths is a Hamiltonian path between $u$ and $v$ in $\overline{C}-\{x,w\}$,  and, hence,  the number of such paths is at least $n^2/2$. Thus,  $G$ has at least $n^2/2\geq c_1^2 n$ Hamiltonian cycles containing both $e$ and $f$. 

Similarly, we can show that   $G$ has at least $n^2/2\geq c_1^2 n$ Hamiltonian cycles containing both $e$ and $f$ if $uv^*, xv^*\in E(C^*)$, or $vv^*, wv^*\in E(C^*)$, or $wv^*, xv^*\in E(C^*)$.

So assume  $vv^*, xv^*\in E(C^*)$. Then the union of $C^*-v^*$ and a Hamiltonian path between $x$ and $v$ in $\overline{C}-\{u,w\}$ is a Hamiltonian cycle in $G$. By the above observation again,  $\bigcup_{i=2}^{t-1}B_i$ has  at least $2^{2\log_2 n-2}=n^2/4$ Hamiltonian paths between $u_1$ and $u_{t-1}$. By applying  Lemma~\ref{tuttepath}, we see that $G[(B_1-u)+v]$ has a Hamiltonian path between $v$ and $u_1$ (and containing $N_G(u)\cap N_G(x))$, and $G[(B_t-w)+x]$ has a Hamiltonian path between $u_{t-1}$ and $x$ (and containing $N_G(w)\cap N_G(v)$). The union of these three  paths is a Hamiltonian path between $v$ and $x$ in $\overline{C}-\{u,w\}$, and there are at least $n^2/4$ of such paths. Hence, $G$ has at least $ n^2/4 \geq c_1^2 n$ Hamiltonian cycles containing both $e$ and $f$.

\medskip

{\it Case} 2.  $|\{i: |V(B_i)|\ge 3\}|< 2\log_2 n$.

Then there exists some integer $k$, with $0\le k\le t-7$, such that $|V(B_i)|=2$ for $i=k, k+1, \ldots, k+7$. Without loss of generality, we may assume $k=0$. 
Then  $u_i$, $1\le i\le 6$, all have degree 4 in $G$. Let $G^*$ be obtained from $G$ by contracting the edge $u_3u_4$ to a vertex, denoted by $u^*$. Then $G^*$ is a  4-connected planar triangulation with $n-1$ vertices and at most $c_1 (n-1)/{\log_2 (n-1)}$ separating 4-cycles.  
By induction, $G^*$ contains at least $c_1^2 (n-1)$ Hamiltonian cycles through both $e$ and $f$. These Hamiltonian cycles in $G^*$ can be modified inside the 4-cycle $u_2 v u_5 x u_2$ to give at least $c_1^2 (n-1)$ Hamiltonian cycles in $G$, all of which use the edge $u_3u_4$. Therefore,  $G$ has at least $c_1^2(n-1)$ Hamiltonian cycles containing  $e,f$, and the edge $u_3u_4$. Hence, to complete the proof of this lemma,  it suffices to find a Hamiltonian cycle in $G$ using $e$ and $f$ but not  the edge $u_3u_4$, as $c_1^2(n-1)+1\geq c_1^2 n$.

Consider $G':=(G^*-u^*)+u_2u_5 $, which is a 4-connected planar triangulation with $n-2$ vertices. Consider the triangles $T_1:=vu_2u_5v$ and $T_2:=xu_2u_5x$ in $G'$. By Lemma~\ref{4edges},  $G'$ has a Hamiltonian cycle $C'$ that contains both $e$ and $f$ as well as edges $e_1\in E(T_1)$ and $e_2\in E(T_2)$, such that $e,f,e_1,e_2$ are all distinct.
We show that $C'$ gives rise to  a Hamiltonian cycle in $G$ containing both  $e$ and $f$ but not the edge $u_3u_4$. By symmetry, we may assume that $e_1=vu_2$ and that  $ e_2=u_2u_5$, or $e_2=u_2x$, or $e_2=u_5x$ but $u_2u_5\notin E(C')$. 

First, suppose $ e_2=u_2u_5$. Then $u_1u_2\notin E(C')$ and, hence, $vu_1\in E(C')$ or $xu_1\in E(C')$. If $vu_1\in E(C')$, then $(C'-\{u_2,v\})\cup u_1u_2u_3vu_4u_5$ is a Hamiltonian cycle in $G$ containing $e$ and $f$ but not $u_3u_4$. If $xu_1\in E(C')$, then $(C'-\{xu_1,u_2\})\cup u_1u_2u_3x\cup vu_4u_5$ is a Hamiltonian cycle in $G$ containing $e,f$ but not $u_3u_4$.

Now suppose $e_2=u_2x$. Then $u_1u_2\notin E(C')$, hence $vu_1\in E(C')$ or $xu_1\in E(C')$.   Note that in this case we have symmetry between $v$ and $x$. Hence, by this symmetry, we may assume $vu_1\in E(C')$. Then $(C'-\{u_2,v\})\cup u_1u_2u_3vu_4x$ is a Hamiltonian cycle in $G$ containing both $e$ and $f$ but not $u_3u_4$.

Finally, suppose $e_2=u_5x$ but $u_2u_5\notin E(C')$. Then  $(C'-\{u_2v, u_5x\})\cup u_2u_3v \cup xu_4u_5$ is a Hamiltonian cycle in $G$ containing $e$ and $f$, but not $u_3u_4$.
\end{proof}

\medskip

We are ready to prove Theorem~\ref{main1}, using Lemma~\ref{2edge} as well as the idea used in its proof. 

\medskip

\begin{proof}[Proof of Theorem~\ref{main1}] Let $c_1= (108\times 16\times 541\times 301\times 2)^{-1}$. We apply induction on $n$, the number of vertices in $G$, to show that $G$ has at least $c_1^4n^2$ Hamiltonian cycles.  It is easy to check that the assertion holds when $n\le 1/ c_1^2+1$, as 
$G$ has at least two Hamiltonian cycles by Lemma~\ref{tuttepath}. So assume that $n\ge 1/c_1^2+2$. 

By Lemma~\ref{special-set}, $G$ has two non-adjacent vertices $v$ and $x$ such that $|N(v)\cap N(x)|> \lfloor 16\log_2 n \rfloor$, or $G$ contains an independent set $S$ of size at least $ c_1 n/\log_2 n$, such that $S$ satisfies (a), (b), (c),  and (d) in Lemma~\ref{special-set} and, hence, (i), (ii), (iii), and (iv) in Lemma~\ref{edgesetF}. In the latter case, it follows from  Lemma~\ref{edgesetF} that $G$ has at least $(3/2)^{|S|}\geq(3/2)^{c_1 n/\log_2 n}\geq c_1^4 n^2$ Hamiltonian cycles. So we may assume that the former occurs. 

Let $C=uvwxu$  such that $N(v)\cap N(x)\subseteq V(\overline{C})$. Note that $\overline{C}$ is a near triangulation.  Moreover, since $G$ is 4-connected, $\overline{C}-\{v,x\}$ has a path from $u$ to $w$, and every block of $\overline{C}-\{v,x\}$ contains an edge of that path. So the blocks of $\overline{C}-\{v,x\}$ can be labelled as $B_1, \dots, B_t$ and the vertices in $N(v)\cap N(x)$ can be labelled as $u_0, u_1, \dots, u_{t-1}, u_t$ such that $V(B_i\cap B_{i+1})=\{u_i\}$ for $1\le i\le t-1$, $B_i\cap B_j=\emptyset$ when $|i-j|\ge 2$, $u_0=u\in V(B_1-u_1)$,  and $u_{t}=w\in V(B_t-u_{t-1})$.

Consider $G':=G-(\overline{C}-C)$ as a near triangulation with $C$ as its outer cycle, which is 3-connected.  By Lemma~\ref{tuttepath}, there exists a $C$-Tutte path $P$ between $u$ and $w$ in $G'$ containing the edge $uv$, which is in fact a  Hamiltonian path in $G'$. To find the desired number of Hamiltonian cycles in $G$ using $P$, we need to find at least $c_1^4n^2$  Hamiltonian paths in $\overline{C}-\{v,x\}$ between $u$ and $w$. 

Observe that,  for  $1\leq j\leq t$, if $|V(B_j)|\ge 3$ then a simple application of Lemma~\ref{tuttepath} shows that $B_j$ has at least two Hamiltonian paths between $u_{j-1}$ and $u_j$.
Hence, if $|\{i: |V(B_i)|\ge 3\}|\ge 2\log_2 n$ then  $\overline{C}-\{v, x\}$ has at least $2^{2\log_2 n}=n^2$ Hamiltonian paths between $u$ and $w$, which, together with $P$, gives at least $n^2\geq c_1^4 n^2$ Hamiltonian cycles in $G$. So assume $|\{i: |V(B_i)|\ge 3\}|< 2\log_2 n$. Then there exists some integer $k$ with $0\leq k\leq t-7$, such that $|V(B_i)|=2$ for $k+1\le i\le k+7$. Since $G$ is 4-connected, $d_G(u_i)=4$  for $k+1\le i\le k+6$. Without loss of generality, we may assume $k=0$.

Let $G^*$ be obtained from $G$ by contracting $u_3u_4$ into a single vertex, say $u^*$, which has degree 4 in $G^*$. Then $G^*$ is a 4-connected planar triangulation with $n-1$ vertices and at most $c_1(n-1)/{\log_2 (n-1)}$ separating 4-cycles. By induction, $G^*$ contains $c_1^4(n-1)^2$ Hamiltonian cycles, each  using exactly two edges incident with $u^*$. It is routine to check that these cycles can be modified inside the 4-cycle $u_2vu_5xu_2$ to give at least $c_1^4(n-1)^2$ Hamiltonian cycles in $G$, all containing the edge $u_3u_4$. 

To obtain additional Hamiltonian cycles in $G$, we consider $H:=(G^*-u^*)+u_2u_5 $, which is a 4-connected planar triangulation with $n-2\ge 1/c_1^2$ vertices and at most $c_1({n-2})/{\log_2 (n-2)}$ separating 4-cycles. 

Note that $vu_2u_5v$ is a facial triangle in $H$, which can be turned into the outer cycle for a different embedding of $H$. 
So by Lemma~\ref{2edge}, $H$ has at least $c_1^2(n-2)$ Hamiltonian cycles through both $vu_2$ and $vu_5$. For each such cycle, say $D$,  we see that  $(D-v)\cup u_2u_3vu_4u_5$ is a Hamiltonian cycle in $G$ not containing the edge $u_3u_4$ (as they use $u_3vu_4$),  and, hence, is different from the Hamiltonian cycles in $G$ obtained previously by modifying those $c_1^4(n-1)^2$ Hamiltonian cycles  in $G^*$. 

 Similarly, since $xu_2u_5x$ is a facial triangle in $H$, we can find at least $c_1^2(n-2)$ new Hamiltonian cycles in $G$ containing $u_3xu_4$.  
Hence, $G$ has at least $c_1^4(n-1)^2+ 2c_1^2(n-2)\geq c_1^4n^2$ Hamiltonian cycles. 
\end{proof}

\section{Planar triangulations with minimum degree 5}
In this section, we consider 4-connected planar triangulations with minimum degree 5, another natural relaxation of 5-connected planar triangulations. Before we present a proof of Theorem~\ref{main2}, we need the following result. 

\begin{lem}\label{special-set-1}
Let $G$ be a 4-connected planar triangulation with $n$ vertices and minimum degree $\delta(G)\geq 5$. Then one of the following holds:
\begin{itemize}
    \item [\textup{(i)}]  $G$ has $2^{\Omega(n^{1/4})}$ Hamiltonian cycles.
    \item [\textup{(ii)}] $G$ has an independent set $S$ of vertices of degree at most 6, such that   $|S|=\Omega (n^{3/4})$ and 
            $S$ saturates no 4-cycle, or 5-cycle, or diamond-6-cycle.
\end{itemize}
\end{lem}
\begin{proof} Let $c>0$ be an arbitrary constant. By Lemma~\ref{4-cycle} (with $t=\lfloor c{n^{1/4}}\rfloor$),   there exist non-adjacent vertices $v,x$ in $G$ such that $|N(v)\cap N(x)|>\lfloor  cn^{1/4} \rfloor$, or  $G$ has an independent set $S_1$ of vertices of degree at most 6, such that $|S_1|\geq n^{3/4}/(108c)$ and $S_1$ saturates no 4-cycle in $G$. If the latter holds then, by Lemmas~\ref{5-cycle} and \ref{diamond}, there exists a subset $S$ of $S_1$, such that $|S|\geq |S_1|/(541\times 301) \geq c' n^{3/4}$, where 
$c'=(541\times 301\times 108c)^{-1}$, 
 and $S$ saturates no 4-cycle, or 5-cycle, or diamond-6-cycle in $G$; so (ii) holds. Thus we may assume that the former occurs.

Let $C=uvwxu$ such that  $N(v)\cap N(x)\subseteq V(\overline{C})$. Since $G$ is 4-connected, $\overline{C}-\{v,x\}$ has a path, say $Q$, between $u$ and $w$. Let $N(v)\cap N(x)=\{u_0, u_1, \ldots, u_k\}$, with $k\ge \lfloor c n^{1/4} \rfloor$, such that $u_0=u, u_k=w$, and 
$u_0, u_1, \ldots, u_k$ occur on $Q$ in order.  Since $G$ is a 4-connected planar triangulation, the blocks of $\overline{C}-\{v,x\}$ can be labelled as $B_1, \ldots,B_k$ such that $u_{i-1},u_i\in V(B_i)$ for $i=1, \ldots, k$, and $B_i\cap B_j=\emptyset$ for $i,j\in \{1, 2,\dots, k\}$ with $|i-j|\geq 2$. 

For each $B_i$ with $|V(B_i)|\ge 3$, $u_{i-1}u_i\notin E(G)$ as $G$ is 4-connected. Thus, if $|V(B_i)|\ge 3$ then $B_i$ is a near triangulation that has no separating triangles and,  by applying Lemma~\ref{tuttepath}, we can find two Hamiltonian paths between $u_{i-1}$ and $u_i$ in $B_i$. 
Since $\delta(G)\ge 5$, we see that $|V(B_i)|\ge 3$ or $|V(B_{i+1})|\ge 3$, for  $i=1, \ldots, k-1$. Thus, $|\{i: |V(B_i)|\ge 3\}|\ge (\lfloor cn^{1/4} \rfloor -1)/2$. It is easy to see that $\overline{C}-\{v,x\}$ has at least $2^{ (\lfloor cn^{1/4} \rfloor -1) /2}$ Hamiltonian paths between $u$ and $w$. 

We view $C$ as the outer cycle of a different embedding of $H: = G- (\overline{C}-C)$. By Lemma~\ref{tuttepath}, $H$ has a $C$-Tutte path $P$ between $u$ and $w$, which is in fact a Hamiltonian path in $H$. Now $P$ and any Hamiltonian path in $\overline{C}-\{v,x\}$ betwwen $u$ and $w$ form a Hamiltonian cycle in $G$. So $G$ has at least $2^{(\lfloor c n^{1/4} \rfloor-1) /2}$ Hamiltonian cycles, and (i) holds.
\end{proof}

Recall the definition of a diamond-4-cycle in Figure~\ref{diamond cycle}, and recall that
the two vertices contained in two triangles in a diamond-4-cycle are its {\it crucial} vertices. In the proof of Theorem~\ref{main2}, we will need to consider
the subgraph of a planar triangulation that lie between two diamond-4-cycles and use the following result on Hamiltonian paths in those subgraphs.

\begin{lem}\label{diamond-4-cycle}
Let $G$ be a near triangulation with outer cycle $C:=uvwxu$ and with no separating triangles, and let 
$z\in V(G)\setminus  V(C)$ have degree 4 in $G$, such that $G[N(z)]$ is contained in a diamond-4-cycle $D'$ in $G-z$, and 
all vertices in $V(G)\setminus (V(C)\cup \{z\}\cup N(z))$ have degree at least 5 in $G$. Suppose  
\begin{itemize}
\item $V(D')\cap V(C)=\emptyset$, or 
\item $V(D')\cap V(C)$ consists of exactly two vertices that are non-adjacent in  $D'\cup C$, and one of these vertices  is a crucial vertex of $D'$, or 
\item $V(D')\cap V(C)$ consists of exactly two vertices that are adjacent in both $D'$ and $C$, and none of these vertices is a crucial vertex of $D'$. 
\end{itemize}
Then one of the following holds:
\begin{itemize}
\item [\textup{(i)}] For any distinct $a,b\in V(C)$, $G-(V(C)\setminus \{a,b\})$ has at least two Hamiltonian paths between $a$ and $b$. 

\item [\textup{(ii)}] There exist distinct $a,b\in V(C)$ such that  $G-(V(C)\setminus \{a,b\})$ has a unique Hamiltonian path, say $P$,  between $a$ and $b$; but for any distinct $c,d\in  V(C)$ with $\{c,d\}\ne \{a,b\}$, $G-(V(C)\setminus \{c,d\})$ has at least two Hamiltonian paths between $c$ and $d$ and avoiding an edge of $P$ incident with $z$.
\end{itemize}
\end{lem}

\begin{proof}
Let $C':=G[N(z)]=u'v'w'x'u'$  and let $y'\in V(G)$ such that $y'u',y'v',y'x'\in E(D')$. Then $y'$ and $u'$ are crucial vertices of $D'$. Without loss of generality assume that $u', v', w',  x'$ occur on $C'$ in clockwise order, and $u, v, w, x$ occur on $C$ in clockwise order.
\medskip

{\it Case} 1. $V(D')\cap V(C)=\emptyset$. 

 Then $z$ is not incident with the infinite face of $G-V(C)$. So for any distinct  $a,b\in V(C)$, $G-(V(C)\setminus \{a,b\})$ cannot be an outer planar graph. Thus, by Lemma~\ref{uw-path} or Lemma~\ref{uv-path}, $G-(V(C)\setminus \{a,b\})$ has  at least two Hamiltonian paths between $a$ and $b$. So (i) holds. 

\medskip

{\it Case} 2.  $V(D')\cap V(C)$ consists of exactly two vertices that are non-adjacent in  $D'\cup C$, and one of these vertices  is a crucial vertex of $D'$.

 Then $V(D')\cap V(C)=\{y', w'\}$, and $V(D')\cap V(C)=\{u,w\}$ or $V(D')\cap V(C)=\{v,x\}$. Without loss of generality, we may assume $u=y'$ and $w=w'$. 

For distinct $a,b\in V(C)$ with $ab\notin E(C)$,  $\{a,b\}=\{u,w\}$ or $\{a,b\}=\{v,x\}$. Since $d_G(v')\geq 5$ and $d_G(x')\geq 5$, $G-(V(C)\setminus \{a,b\})$ cannot be a path. Therefore, by Lemma~\ref{uw-path}, $G-(V(C)\setminus \{a,b\})$ has at least two Hamiltonian paths between $a$ and $b$. 

Now we consider $a,b\in V(C)$ with $ab\in E(C)$. Then, $u'$ or $z$ is not incident with the infinite face of $G-(V(C)\setminus \{a,b\})$; so $G-(V(C)\setminus \{a,b\})$ cannot be an outer planer graph. Hence, by Lemma~\ref{uv-path}, $G-(V(C)\setminus \{a,b\})$ has at least two Hamiltonian paths between $a$ and $b$. 

\medskip

{\it Case} 3.   $V(D')\cap V(C)$ consists of exactly two vertices that are adjacent in both $D'$ and $C$, and none of these vertices is a crucial vertex of $D'$. 

 Then $V(D')\cap V(C)=\{v',w'\}$ or  $V(D')\cap V(C) =\{x',w'\}$. By the symmetry among the edges in $C$ and between the two orientations of $C$, we may assume  $V(D')\cap V(C)=\{v,w\}$.
Further by the symmetry between $v'$ and $x'$, we may assume that $v=v'$ and $w=w'$. 

\begin{itemize}
   \item []Claim 1. For any distinct $a,b\in V(C)$ with $\{a,b\}\ne \{u,x\}$,  $G-(V(C)\setminus \{a,b\})$ has at least   two Hamiltonian paths between $a$ and $b$.
\end{itemize}
If  $\{a,b\}=\{u,w\}$ or $\{a,b\}=\{v,x\}$ then the claim follows from Lemma~\ref{uw-path} (as $d_G(y')\ge 5$). 
If $\{a,b\}=\{v,w\}$ or $\{a,b\}=\{v,u\}$ then $u'$ is not incident with the infinite face of $G-(V(C)\setminus\{a,b\})$; 
so the claim follows from Lemma~\ref{uv-path}. 

Now suppose $\{a,b\}=\{w,x\}$. Suppose $G-(V(C)\setminus \{w,x\})$, i.e., $G-\{u,v\}$, has exactly one Hamiltonian path between $w$ and $x$. Then by Lemma~\ref{uv-path}, $G-\{u,v\}$ is an outer planar graph. 
Now $x'$ is incident with the infinite face of $G-\{u,v\}$; so $x'u,x'x\in E(G)$. Also, $y'$ is incident with the infinite face of $G-\{u,v\}$; so  $y'u\in E(G)$. Then, since $d_G(y')\ge 5$, $ux'y'u$ or $uy'vu$ is a separating triangle in $G$, a contradiction. Thus, we have Claim 1. 

\medskip

Therefore, if $G-(V(C)\setminus \{u,x\})$, i.e., $G-\{v,w\}$, has two Hamiltonian paths between $u$ to $x$, then (i) follows from Claim 1. 
Hence, we may assume that 
$H:=G-\{v,w\}$ has at most one Hamiltonian paths between  $u$ and $x$. Then by Lemma~\ref{uv-path},  $H$ is an outer planar graph and the unique Hamiltonian path $P$ in $H$ between $u$ and $x$ contains $u'zx'$, since $z\in N_G(v)\cap N_G(w)$, $V(uPz)\subseteq N_G(v)$ and $V(zPx)\subseteq N_G(w)$. 

Note that $d_{G}(u')=d_{G}(z)=4$ and $u'x'zu'$ is a triangle in $H-\{u,x\}=G-C$. 

\begin{itemize}
    \item [] Claim 2. $H-u$ has two Hamiltonian paths $P_1,P_2$ between $z$ and $x$, and $H-x$ has two Hamiltonian paths $Q_1,Q_2$ between $z$ and $u$. 
\end{itemize}
We only consider $H-u$, as the case for $H-x$ can be taken care of by the same argument.  Suppose $H-u$ is 2-connected. Then let $F$ denote the outer cycle of $H-u$. By Lemma~\ref{tuttepath}, $H-u$ contains $F$-Tutte paths $P_1,P_2$ between $z$ and $x$ such that $zu'\in E(P_1)$ and 
$zx'\in E(P_2)$. We claim that $P_1,P_2$ are in fact Hamiltonian paths in $H-u$. For suppose otherwise, and let $B$ be a $P_i$-bridge of $H-u$ with $B-P_i\ne \emptyset$ for some $i\in \{1,2\}$. Recall that $P-u\subseteq F$,  $V(uPz)\subseteq N_G(v)$, and $V(zPx) \subseteq N_G(w)$. To avoid separating triangles in $G$, the unique vertex in $N_P(u)$, say $u^*$, is in $B-P_i$. Hence, we see that $d_G(u^*)=4$, a contradiction.

Now assume that $H-u$ is not 2-connected. Then since $u',z$ are the only vertices in $V(G)\setminus V(C)$ with degree 4 in $G$, $x$ has a unique neighbor in $H-u$, say $x^*$, and $(H-u)-x$ is 2-connected. Moreover, $x^*\ne x'$ as, otherwise, $y'$ would have degree 4 in $G$. 
Let $F$ denote the outer cycle of $(H-u)-x$. By Lemma~\ref{tuttepath}, $(H-u)-x$ contains $F$-Tutte paths $R_1,R_2$ between $z$ and $x^*$ such that $zu'\in E(R_1)$ and 
$zx'\in E(R_2)$. Note that $P-\{u,x\}\subseteq F$, $V(uPz) \subseteq N_G(v)$, and $V(zPx) \subseteq N_G(w)$. As in the previous paragraph,  we see that
$R_1,R_2$ are Hamiltonian paths in $(H-u)-x$, since $G$ has no separating triangles. So
 $R_1\cup x^*x ,R_2\cup x^*x$ are Hamiltonian paths between $z$ and $x$ in $H-u$. This completes the proof of Claim 2. 

\medskip
We now show that (ii) holds, with $\{a,b\}=\{u,x\}$. Let $c, d\in V(C)$ be distinct such that $\{c,d\}\neq \{u,x\}$.
If $\{c,d\}=\{u,w\}$ then  $Q_1\cup zw,Q_2\cup zw$ are distinct Hamiltonian paths in $G-\{v,x\}=G-(V(C)\setminus \{c,d\})$  between $u$ and $w$ and not containing $u'zx'$.   If $\{c,d\}=\{v,x\}$ then $P_1\cup zv,P_2\cup zv$ are distinct Hamiltonian paths in $G-\{u,w\}=G-(V(C)\setminus \{c,d\})$  between $v$ and $x$  and not containing $u'zx'$. If $\{c,d\}=\{u,v\}$ then  $Q_1\cup zv, Q_2\cup zv$ are distinct Hamiltonian paths in $G-\{w,x\}=G-(V(C)\setminus \{c,d\})$ between $u$ and $v$ and not containing  $u'zx'$. If $\{c,d\}=\{w,x\}$ then $P_1\cup zw, P_2\cup zw$ are distinct Hamiltonian paths  in $G-\{u,v\}=G-(V(C)\setminus \{a,b\})$ between $w$ and $x$ and not containing $u'zx'$.  

Thus to prove that (ii) holds, it remains to consider $\{c,d\}=\{v,w\}$. Then $G-(V(C)\setminus \{c,d\})=G-\{u,x\}$. Using the fact that $H=G-\{v,w\}$ is outer planer with $P$ as its unique Hamiltonian path between $u$ and $x$, we see that  $G-\{u,x,w\}$ has a Hamiltonian path $R_1$ between $z$ and  $v$, and $G-\{u,x,v\}$ has a Hamiltonian path $R_2$ between $z$ and $w$. Hence, $R_1\cup zw$ and $R_2\cup zv$ are distinct Hamiltonian paths between $v$ and $w$ in $G-\{u,x\}$ and not containing $u'zx'$. 
\end{proof}

\medskip

\begin{proof}[Proof of Theorem~\ref{main2}]

By Lemma~\ref{special-set-1}, we may assume that, for some constant $c>0$, 
\begin{itemize}
\item [(1)] $G$ has an independent set $S$ of vertices of degree 5 or 6, such that $|S|\geq c n^{3/4}$ and $S$ saturates no 4-cycle, or 5-cycle, or diamond-6-cycle in $G$.  
\end{itemize}

Let $S^*$ consist of all vertices in $S$ that are each adjacent to exactly 3 vertices of a separating 4-cycle in $G$. We may assume that 
\begin{itemize}

\item [(2)] $|S^*|\ge c n^{3/4} /2$. 
\end{itemize}
For, otherwise, $|S\setminus S^*|\ge c n^{3/4}/2$. Then, since $\delta(G)\ge 5$,  we can apply Lemma~\ref{edgesetF} to $G, S\setminus S^*$ and conclude that $G$ has at least $(3/2)^{|S\setminus S^*|}\geq (3/2)^{c n^{3/4}/2}>2^{cn^{1/4}/2}$ Hamiltonian cycles, and the assertion of Theorem~\ref{main2} holds. 
So we may assume (2).

\medskip

 For convenience, we use $\overline{D}$, for any diamond-4-cycle $D$, to denote the subgraph of $G$ consisting of vertices and edges of $G$ in the closed disc bounded by the outer cycle of $D$. 

We now define a collection ${\cal D}$ of diamond-4-cycles that are associated with vertices in $S^*$, one for each vertex in $S^*$. 
 For each  $v\in S^*$, $v$ is adjacent to three vertices of some separating  4-cycle, say $C_v$, and, since $\delta(G)\geq 5$, $D_v: =G[C_v+v]$ is a diamond-4-cycle in $G$. Note that $v$ is a crucial vertex of $D_v$. For every $v\in S^*$, we choose $D_v$ so that $\overline{D_v}$ is maximal.
 Now let ${\cal D}=\{D_v: v\in S^*\}$. 
Clearly, if $u\ne v \in S^*$ then $D_u\ne D_v$ (as $S^*$ does not saturate any 4-cycle in $G$). Thus $|{\cal D}|=|S^*|\geq c n^{3/4}/2$.  We now prove the following claim.

\begin{itemize}
\item [(3)] For distinct $u,u'\in S^*$, $|V(D_u)\cap V(D_{u'})|\le 2$. Moreover, if 
$V(D_u)\cap V(D_{u'})$ consists of two vertices, say $a$ and $b$,  then either $a b\in E(D_u)\cap E(D_{u'})$ and neither $a$ nor $b$ is a crucial vertex of $D_u$ or $D_{u'}$, or 
$ab \notin  E(D_u)\cup E(D_{u'})$ and $D_u, D_{u'}$ each have precisely one crucial vertex in $\{a,b\}$.
\end{itemize}
To prove (3), let $u, u' \in S^*$ be distinct. Then $u'\notin V(D_u)$ as otherwise $u,u'$ are contained in a 4-cycle in $D_u$, a contradiction as $S^*$ saturates no 4-cycle in $G$. Similarly, $u\notin V(D_{u'})$.  Moreover,  $|N(u')\cap V(D_u)|\le 1$; for otherwise $u,u'$ are contained in a 4-cycle or 5-cycle in $G[D_u+u']$, a contradiction as $S^*$ saturates no 4-cycle or 5-cycle in $G$. Likewise, $|N(u)\cap V(D_{u'})|\le 1$. Therefore, $|V(D_u)\cap V(D_{u'})|\le 2$. 

Now suppose $V(D_u)\cap V(D_{u'})=\{a,b\}$ with $a\ne b$. If $ab\in E(D_u)\setminus E(D_{u'})$ then $ab$ and two edges in $D_u'$ form a separating triangle in $G$, 
a contradiction. So $ab\notin E(D_u)\setminus E(D_{u'})$. Similarly, $ab\notin E(D_{u'})\setminus E(D_{u})$. 

If $ab\in E(D_u)\cap E(D_{u'})$ then neither $a$ nor $b$ is a crucial vertex in $D_u$ or $D_{u'}$, as otherwise $u$ and $u'$ would be contained in 4-cycle or 5-cycle in $G$. 
If $ab\notin E(D_u)\cup E(D_{u'})$ then $D_u$ and $D_{u'}$ each have exactly one crucial vertex in  $\{a,b\}$, to avoid a 4-cycle or 5-cycle containing $\{u,u'\}$. 
This completes the proof of (3).


\medskip

By (3), for any $D_1,D_2\in {\cal D}$, either $\overline{D_1}-D_1$ and $\overline{D_2}-D_2$ are disjoint, or $\overline{D_1}$ contains $\overline{D_2}$ or vice versa. We may assume that 
\begin{itemize}
    \item [(4)] there exists an integer $t\ge n^{1/2}$ and $D_1, D_2,\dots , D_t\in {\cal D}$, such that $\overline{D_1}\supseteq \overline{D_2 }\supseteq \dots \supseteq \overline{D_t}$.
\end{itemize}
For, otherwise, there exists some integer $k\ge c n^{1/4} /2$ and diamond-4-cycles $D_1, \ldots, D_k\in \cal{D}$ such that $\overline{D_i}-D_i$ and $\overline{D_j}-D_j$ are 
disjoint whenever $1\leq i\ne j \leq k$. Let $G^*$ be obtained from $G$ by contracting $\overline{D_i}-D_i$ to a new vertex $v_i$, for all $1\le i\le k$. 
Then $G^*$ is a 4-connected planar triangulation and, hence, has a Hamiltonian cycle, say $C$. 

Let $a_i,b_i\in N_{G^*}(v_i)$ such that $a_iv_ib_i\subseteq C$ for $1\le i\le k$. Let $C_i$ denote the  4-cycle in $G$, such that $\overline{D_i}-D_i$ is the interior of $C_i$. Then, $C_i\subseteq D_i$ and $\overline {C_i}-C_i=\overline {D_i}-D_i$.
Since all vertices in $\overline{C_i}-C_i$ have  degree at least 5, $\overline{C_i}-(V(C_i)\setminus \{a_i,b_i\})$ cannot be outer planar. 
So by Lemma~\ref{uw-path} or Lemma~\ref{uv-path},  $\overline{C_i}-(V(C_i)\setminus \{a_i,b_i\})$ has at least two Hamiltonian paths between $a_i$ and $b_i$. 

We can form a Hamiltonian cycle in $G$ by taking the union of $C-\{v_i: 1\le i\le k\}$ and by selecting one Hamiltonian path  between $a_i$ and $b_i$ in each $\overline{C_i}-(V(C_i)\setminus \{a_i,b_i\})$ for $1\leq i\leq k$. Thus, $G$ has at least $2^k \ge 2^{cn^{1/4}/2}$ Hamiltonian cycles, completing the proof of (4). 

\medskip
For each $1\leq j\leq t$, let $u_j$ and $y_j$ be the crucial vertices of $D_j$ and let $v_j,w_j,x_j$ be the other vertices of $D_j$, so that $y_j v_j w_j x_j y_j$ is the outer cycle of $D_j$ and $y_j, v_j, x_j\in N_G(u_j)$. Then $u_j\in S^*$ or $y_j\in S^*$. Let $C_j:= u_j v_j w_j x_j u_j$. Then by (4),  $\overline{C_1}\supseteq \overline{C_2}\supseteq\dots \supseteq \overline{C_t}$. For $j=1, \ldots, t-1$, 
let $G_j$ denote the graph obtained from $\overline{C_j}$ by contracting $\overline{C_{j+1}}-C_{j+1}$ to a new vertex, denoted by $z_{j+1}$. Note that $G_j$ is a near triangulation with outer cycle $C_j$ and that $G_j$ contains the diamond-4-cycle $D_{j+1}$. For convenience, let $G_t:=\overline{C_t}$. We claim that,

\begin{itemize}
\item [(5)] for any $j\in \{1, \ldots, t-1\}$, if  $|V(C_j)\cap V(D_{j+1})|=1$ then, for any distinct $a,b\in V(C_j)$, $G_j-(V(C_j)\setminus \{a,b\})$ has two Hamiltonian paths between $a$ and $b$. 
\end{itemize}
For, suppose $V(C_j)\cap V(D_{j+1})=\{v\}$. If $v=y_{j+1}$ then $z_{j+1}$ is not incident with the infinite face of $G_j-C_j$; so $G_j-C_j$ is not outer planar and, by Lemma~\ref{uw-path} or Lemma~\ref{uv-path}, for any distinct $a,b\in V(C_j)$, $G_j-(V(C_j)\setminus \{a,b\})$ has two Hamiltonian paths between $a$ and $b$.
If $v=w_{j+1}$ then $u_{j+1}$ is not incident with the infinite face of $G_j-C_j$; so again by Lemma~\ref{uw-path} or Lemma~\ref{uv-path},  for any distinct $a,b\in V(C_j)$, $G_j-(V(C_j)\setminus \{a,b\})$ has two Hamiltonian paths between $a$ and $b$.
Thus, we may assume without loss of generality that  $v=v_j=v_{j+1}$.

For $\{a,b\}=\{u_j, w_j\}$ or $\{a,b\}=\{v_j, x_j\}$,  since $d_{G_j}(y_{j+1})=d_G(y_{j+1})\geq 5$, it follows from Lemma~\ref{uw-path} that  $G_j-(V(C_j)\setminus \{a,b\})$ has two Hamiltonian paths between $a$ to $b$. For $\{a,b\}=\{v_j, u_j\}$ or $\{a,b\}=\{v_j, w_j\}$, since $z_{j+1}$ is not incident with the infinite face of $G_j-(V(C_j)\setminus \{a,b\})$, $G_j-(V(C_j)\setminus \{a,b\})$ has two Hamiltonian paths between $a$ and $b$ by Lemma~\ref{uv-path}. For $\{a,b\}=\{w_j, x_j\}$, since $d_{G_j}(y_{j+1})\geq 5$, $x_{j+1}$ is not incident with the infinite face of $G_j-(V(C_j)\setminus \{a,b\})$; it follows from Lemma~\ref{uv-path} that $G_j-(V(C_j)\setminus \{a,b\})$ has two Hamiltonian paths between $a$ and $b$. 

Finally, consider $\{a,b\}=\{u_j, x_j\}$. Suppose $G_j-(V(C_j)\setminus \{a,b\})=G_j-\{v_j,w_j\}$ has a unique Hamiltonian path between $a$ and $b$. Then $G_j-\{v_j,w_j\}$  is outer planar; so $x_{j+1}$ is incident with the infinite face of $G_j-\{v_j,w_j\}$ and, hence, $x_{j+1}w_j\in E(G)$. 
 Therefore, $D_{j+1}':=(D_{j+1}-w_{j+1}) \cup  v_{j+1}w_jx_{j+1}$ is a diamond-4-cycle containing $\{u_{j+1},y_{j+1}\}$, and 
$\overline{D_{j+1}'}$ properly contains $\overline{D_{j+1}}$. Thus $D'_{j+1}$ contradicts the choice  of $D_{j+1}$. So $G_j-(V(C_j)\setminus \{a,b\})$ has at least two Hamiltonian paths between $a$ and $b$.  This completes the proof of (5). 

\medskip

We may assume that  
\begin{itemize}
    \item [(6)] there exists some integer $k$ with  $1\leq k\leq t-n^{1/4}$ and there exist distinct $a_j,b_j\in V(C_j)$ for $k \leq j\leq k+\lfloor n^{1/4} \rfloor$, such that, for $k\leq j\leq k+ \lfloor n^{1/4} \rfloor -1$, $G_{j}-(V(C_{j})\setminus \{a_j,b_j\})$ has a unique Hamiltonian path $P_j$ between $a_j$ and $b_j$,  and $P_j$ contains $a_{j+1}z_{j+1}b_{j+1}$.
\end{itemize}
For $j=1, \ldots, t$, let $H_j$ denote the graph obtained from $G$ by contracting $\overline{C_j}-C_j$ to a new vertex $z_j$. Let $H_{t+1}:=G$. Note that $H_j$, $1\leq j\leq t+1$, are 4-connected planar triangulations.

We see that $H_1$ has a Hamiltonian cycle, say $F_1$, and let $a_1,b_1\in V(C_1)$ such that $a_1z_1b_1\subseteq F_1$. We now define a rooted tree $T$ whose root $r$ represents $F_1$, and whose leaves are Hamiltonian cycles in $G$. Recall the graphs $G_j$, $1\leq j \leq t$.

For each Hamiltonian path $P_1$ in $G_1-(V(C_1)\setminus \{a_1,b_1\})$ between $a_1$ and $b_1$, $F_2:=(F_1-z_1)\cup P_1$ is a Hamiltonian cycle in $H_2$;  we add a neighbor to $r$ in $T$ to represent $F_2$.  This defines all vertices of $T$ at distance 1 from the root $r$.   Now, suppose we have defined all vertices of $T$ at distance $s$ from $r$, for some $s$ with $1\le s\le t-1$, each of which represents a Hamiltonian cycle in $H_{s+1}$.  To define the 
vertices of $T$ at distance $s+1$ from $r$, we let $v$ be an arbitrary vertex in $T$ that is at distance $s$ from $r$. Then $v$ represents a Hamiltonian cycle $F_{s+1}$ in $H_{s+1}$. Let $a_{s+1},b_{s+1}\in V(F_{s+1})$ such that $a_{s+1}z_{s+1}b_{s+1}\subseteq F_{s+1}$. For each  Hamiltonian path $P_{s+1}$ in $G_{s+1}-(V(C_{s+1})\setminus \{a_{s+1},b_{s+1}\})$ between $a_{s+1}$ and $b_{s+1}$, $F_{s+2}:=(F_{s+1}-z_{s+1})\cup P_{s+1}$ is a Hamiltonian cycle in $H_{s+2}$;  we add a neighbor to $v$ in $T$ to represent $F_{s+2}$. This process continues until $s=t-1$. Then the leaves of $T$ correspond to 
distinct Hamiltonian cycles in $H_{t+1}=G$. Note that, by construction, the distance in $T$ between the root and any leaf is $t$.

If $T$ has a path of length $\lfloor n^{1/4}\rfloor $ whose internal vertices are of degree 2 in $T$, then (6) holds. So assume no such path exists in $T$. We obtain the tree $T^*$ from $T$ by contracting all edges of $T$ incident with degree 2 vertices in $T$. Then all vertices in $T^*$, except the leaves and possibly the root, have degree at least 3. Since each leaf of $T$ has distance $t\ge n^{1/2}$ from the root $r$, 
the distance between the root and any leaf in $T^*$ is at least $n^{1/4}$. Hence, $T^*$ and, thus, $T$  both have at least $2^{n^{1/4}}$ leaves.  Therefore,  $G$ has at least $2^{n^{1/4}}$ Hamiltonian cycles. This proves (6). 

\medskip 

Without loss of generality, we may assume $k=1$ in (6), i.e., for each $1 \leq j\leq 1+\lfloor n^{1/4}\rfloor$, there exist $a_j,b_j\in V(C_j)$ such that, for $1\leq j \leq \lfloor n^{1/4}\rfloor$,  $G_{j}-(V(C_{j})\setminus \{a_j,b_j\})$ has a unique Hamiltonian path $P_j$ between $a_j$ and $b_j$,  and $P_j$ contains $a_{j+1}z_{j+1}b_{j+1}$. For the sake of simplicity, let $q:=1+\lfloor n^{1/4}\rfloor$.

\medskip

By  (3) and (5), we see that $D_{j+1}$, $C_j$, and $G_j$, for $1\leq j\leq q-1$, satisfy the conditions in Lemma~\ref{diamond-4-cycle} (with $D_{j+1}, C_j, G_j$ as $D', C, G$, respectively,  in Lemma~\ref{diamond-4-cycle}). 
Hence, by Lemma~\ref{diamond-4-cycle}, we know that, for $ 1\leq j\leq  q-1$,  if $\{c_j, d_j\}\subseteq V(C_j)$ and $\{c_j,d_j\}\neq \{a_j,b_j\}$, then  $G_j-(V(C_j)\setminus \{c_j,d_j\})$ has at least  two Hamiltonian paths between $c_j$ and $d_j$ but not containing $a_{j+1}z_{j+1}b_{j+1}$. Recall that for $1\leq j \leq q$, $H_j$ denotes the graph obtained from $G$ by contracting $\overline{C_j}-C_j$ to a new vertex $z_j$.

Let $c_{1},d_{1}\in V(C_{1})$ be distinct  such that $\{c_{1},d_{1}\}\ne \{a_{1},b_{1}\}$. Then by Lemma~\ref{tuttepath}, 
 $G-(\overline{C_{1}}-C_{1})$ has a Hamiltonian path $Q$ between $c_{1}$ and $d_{1}$. In $H_{1}$, $F_{1}= Q\cup c_{1}z_{1} d_{1}$ is a Hamiltonian cycle. 
We now define a rooted tree $T$ whose root $r$ represents  $F_{1}$, and whose leaves are Hamiltonian cycles in $G$. 

 By (6) (where we assume $k=1$), $G_{1}-(V(C_{1})\setminus \{c_{1},d_{1}\})$ has at least two Hamiltonian paths between $c_{1}$ and $d_{1}$ and not containing $a_{2}z_{2}b_{2}$. 
For each such Hamiltonian path $P_1$, we see that $(F_{1}-z_{1})\cup P_1$ is a Hamiltonian cycle in $H_{2}$, and we add a vertex to $T$ representing $(F_1-z_1)\cup P_1$ and make it adjacent to $r$. This defines all vertices of $T$ within distance 1 from $r$.  Note that $d_T(r)\ge 2$. Now suppose we have defined the vertices of $T$ at distance $s$ from $r$ for some $s$ with $1\leq s<q-1$, each representing a Hamiltonian cycle in $H_{s+1}$ not containing $a_{s+1}z_{s+1}b_{s+1}$. 
To define the vertices of $T$ that are at distance $s+1$ from $r$,  let $v$ be an arbitrary vertex of $T$ at  distance $s$ from $r$. 
Then $v$ corresponds to a Hamiltonian cycle $F_{s+1}$ in $H_{s+1}$ not containing $a_{s+1}z_{s+1}b_{s+1}$. Let $c_{s+1},d_{s+1}\in V(C_{s+1})$ be distinct such that $c_{s+1}z_{s+1}d_{s+1}\subseteq F_{s+1}$. Then $\{c_{s+1},d_{s+1}\}\ne \{a_{s+1},b_{s+1}\}$. Hence, by (6), $G_{s+1}-(V(C_{s+1})\setminus \{c_{s+1},d_{s+1}\})$ has at least two Hamiltonian paths between $c_{s+1}$ and $d_{s+1}$ and not containing $a_{s+2}z_{s+2}b_{s+2}$. For each such path $P_{s+1}$,  $(F_{s+1}-z_{s+1})\cup P_{s+1}$ is a Hamiltonian cycle in $H_{s+2}$ not containing $a_{s+2}z_{s+2}b_{s+2}$, and we add a neighbor to $v$ in $T$ to represent $(F_{s+1}-z_{s+1})\cup P_{s+1}$. Thus, $d_T(v)\ge 3$. We repeat this process for $s=1, \ldots, q-2$. 

For an arbitrary vertex $u$ of $T$ that has distance $q-1$ from $r$ in $T$, it represents a Hamiltonian cycle $F_{q}$ in $H_{q}$. Assume $c_{q}z_{q}d_{q}\subseteq F_{q}$. Since  $\delta(G)=5$, we may apply Lemma~\ref{uw-path} or Lemma~\ref{uv-path} to conclude that  $\overline{C_{q}}-(V(C_{q})\setminus\{c_{q}, d_{q}\})$ has at least two Hamiltonian paths between $c_{q}$ and $d_{q}$. For each such path $P_q$, $(F_{q}-z_{q})\cup P_q$ is a Hamiltonian cycle in $G$, and we add a neighbor to $u$ in $T$ and this vertex is a leaf of $T$. Thus, $d_T(u)\geq 3$. The distance between the root and any leaf in $T$ is $q$. Moreover, for any vertex $w\in V(T)$ which is not the root or a leaf, $d_T(w)\geq 3$. So $T$ has at least $2^{q}\geq 2^{n^{1/4}}$ leaves. Hence, $G$ has at least $2^{n^{1/4}}$ Hamiltonian cycles.
\end{proof}

\section*{Acknowledgements}
We thank the anonymous referees for their careful reading of the manuscript and thoughtful suggestions.

\end{document}